\documentclass{article}
\usepackage{amsmath,amsthm,amsopn,amstext,amscd,amsfonts,amssymb,verbatim}
\usepackage{color}
\usepackage{natbib}
\usepackage{subfigure,graphicx}

\def \build#1#2#3{\mathrel{\mathop{#1}\limits^{#2}_{#3}}}
\def\tr{\mathop{\rm tr}\nolimits}
\def\E{\mathop{\rm E}\nolimits}
\def\diag{\mathop{\rm diag}\nolimits}
\def\cov{\mathop{\rm Cov}\nolimits}

\def\vec{\mathop{\rm vec}\nolimits}

\newcommand {\bgtext}[1] {\boldmath
             \(#1\)\unboldmath}
\newcommand {\bg}[1]% Negritas (modo matem\'{a}tico)
             {\mbox{\bgtext{#1}}
            }

\renewenvironment{abstract}
                 {\vspace{6pt}
                  \begin{center}
                  \begin{minipage}{5in}
                  \centerline{\textbf{Abstract}}
                  \noindent\ignorespaces
                 }
                 {\end{minipage}\end{center}}
\newtheorem{theorem}{\textbf{Theorem}}[section]

\theoremstyle{definition}

\newtheorem{example}{\textbf{Example}}[section]

\title{\Large \textbf{Tests about $R$ multivariate simple linear models}}
\author{
  \textbf{Jos\'e A. D\'{\i}az-Garc\'{\i}a} \thanks{Corresponding author\newline
    }\\
  {\normalsize Universidad Aut\'onoma de Chihuahua} \\
  {\normalsize Facultad de Zootecnia y Ecolog\'{\i}a} \\
  {\normalsize Perif\'erico Francisco R. Almada Km 1, Zootecnia} \\
  {\normalsize 33820 Chihuahua, Chihuahua, M\'exico}\\
  {\normalsize E-mail: jadiaz@uach.mx}\\
  \textbf{Oscar Alejandro Mart\'{\i}nez-Jaime}\\
  {\normalsize Universidad de Guanajuato} \\
  {\normalsize Departamento de Agronom\'{\i}a} \\
  {\normalsize Divisi\'on Ciencias de la Vida, Campus Irapuato-Salamanca} \\
  {\normalsize Ex-Hacienda "El Copal", Km. 9; carretera Irapuato-Silao} \\
  {\normalsize 36500 Irapuato, Guanajuato, M\'exico} \\
  {\normalsize E-mail: oscarja@ugto.mx} \\[2ex]
}
\date{}
\begin{document}
\maketitle

\begin{abstract}
Hypothesis about the parallelism of the regression lines in $R$ multivariate simple linear models are
studied in this paper. Tests on common intercept and sets of lines intersected at a fixed value, are also
developed. An application in an agricultural entomology context is provided.
\end{abstract}

\section{Introduction}\label{sec1}

There a number of research studies involving the behavior of a dependent variable $Y$ as a  function of
one independent variable $X$. Sometimes the experiment accepts a simple linear model and usually this
model is proposed for different experimental or observational situations. Then the following situations
can emerge: the simple linear models have the same intercept; or these lines are parallel; or given a
particular value of the independent variable $X$,  say $x_{0}$, the lines are intersected in such value.

We illustrate these situations through the following examples:

\begin{example}
Several diets are used to feed goats in order to determine the effect for losing or gaining weight. Three
goat breeds are used, and for each breed the relationship between the gain or loss of weight in pounds
per goat $Y$ and the amount of diet in pounds ingested for each goat $X$ is given by
$$
  y_{1j}= \alpha_{1}+\beta_{1}x_{1j}+\epsilon_{1j} \qquad y_{2k}= \alpha_{2}+\beta_{2}x_{2k}+\epsilon_{2k}
  \quad y_{3r}= \alpha_{3}+\beta_{3}x_{3r}+\epsilon_{3r},
$$
$j = 1, 2, \dots, n_{1}$, $k = 1, 2, \dots, n_{2}$, $r = 1, 2, \dots, n_{3}$, $n_{s} \geq 2$, $s =
1,2,3$. The investigator claims for a parallelism of the lines, that is, if
$\beta_{1}=\beta_{2}=\beta_{3}$ (if the increase in the average weight of each goat per unit of diet is
the same for all breeds). Or the researcher can ask for equality in the intercepts, that is, if
$\alpha_{1} = \alpha_{2} = \alpha_{3}$ (if the average weight of each goat breed is the same when all
breeds are fed with the same diet).
\end{example}

\begin{example}
An essay is carried out to study the relationship of the age $X$ and the cholesterol content in blood $Y$ of
individuals between 40 and 55 years of age. In this situation, a simple linear model is assumed, but in
the essay is considered female and male individuals, then, it is proposed that a model for each sex is
more appropriate. The models are
$$
  y_{1j}= \alpha_{1}+\beta_{1}x_{1j}+\epsilon_{1j}, \quad\mbox{ for famales},
$$
and
$$
  y_{2k}= \alpha_{2}+\beta_{2}x_{2k}+\epsilon_{2k}, \quad\mbox{ for males;}
$$
$j = 1, 2, \dots, n_{1}$, $k = 1, 2, \dots, n_{2}$, $n_{s} \geq 2$, $s = 1,2$. The investigator wants to know
 if $\alpha_{1}+\beta_{1}x_{0}=\alpha_{2}+\beta_{2}x_{0}$ (if at age $x_{0}$ the
cholesterol content in blood is the same for female and male individuals).
\end{example}

However, more realistic situations ask for the behavior of more than one dependent variable $\mathbf{y}'
= (y_{1},\dots,y_{R})$ as a function of an independent variable  $X$. In the statistical modeling of such
situations, the \textit{multivariate simple linear model} appears as an interesting alternative. In a
wider context, the research can ask the same preceding hypothesis about the parallelism of a set of
lines, or the same intercept, or a common given intersection point.

Some preliminary results about matrix algebra, matrix multivariate distributions and general multivariate
linear model are showed, see Section \ref{sec:2}. By using likelihood rate and union-intersection
principles, Section \ref{sec:3} derive the multivariate statistics versions for the above mentioned
hypotheses: same intercept, parallelism and intersection in a known point. Also, these results are
extended to the elliptical case when the $x$'s are fixed or random. Section \ref{sec:4} applies the
developed theory in the context of agricultural entomology .

\section{Preliminary results}\label{sec:2}

A detailed discussion of the univariate linear model and related topics may be found in \citet{g:76} and
\citet{ds:81} and for the multivariate linear model see \citet{re:95} and \citet{s:84}, among many others.
For completeness, we shall introduce some notations, although in general we adhere to standard notation
forms.

\subsection{Notation, matrix algebra and matrix multivariate distribution.}
For our purposes: if $\mathbf{A}\in \Re^{n \times m}$ denotes a matrix, this is, $\mathbf{A}$ have $n$
rows and $m$ columns, then $\mathbf{A}'\in \Re^{m \times n}$ denotes its transpose matrix, and if
$\mathbf{A}\in \Re^{n \times n}$ has an inverse, it shall be denoted by $\mathbf{A}^{-1} \in \Re^{n
\times n}$. An identity matrix shall be denoted by $\mathbf{I}\in \Re^{n \times n}$, to specified the
size of the identity, we will use $\mathbf{I}_{n}$. A null matrix shall be denoted as $\mathbf{0} \in
\Re^{n\times m}$. A vector of ones shall be denoted by $\mathbf{1}\in \Re^{n}$. For all matrix
$\mathbf{A}\in \Re^{n \times m}$ exist $\mathbf{A}^{-} \in \Re^{m\times n}$ which is termed Moore-Penrose
inverse. Similarly given $\mathbf{A}\in \Re^{n \times m}$ exist $\mathbf{A}^{c} \in \Re^{m \times n}$
such that $\mathbf{A}\mathbf{A}^{c}\mathbf{A} = \mathbf{A}$, $\mathbf{A}^{c}$ is termed conditional
inverse. It is said $\mathbf{A}\in \Re^{n \times n}$ is symmetric matrix if $\mathbf{A} = \mathbf{A}'$
and if all their eigenvalues are positive the matrix $\mathbf{A}$ is said to be positive definite, which shall be
denoted as $\mathbf{A} > \mathbf{0}$. If $\mathbf{A} \in \Re^{n \times m}$ is writing in terms of its $m$
columns, $\mathbf{A} = (\mathbf{A}_{1}, \mathbf{A}_{2}, \dots, \mathbf{A}_{m})$, $\mathbf{A}_{j} \in
\Re^{n}$, $j = 1, 2 \dots, m$, $\vec(\mathbf{A}) \in \Re^{nm}$ denotes the vectorization of $\mathbf{A}$,
moreover, $\vec'(\mathbf{A}) = (\vec(\mathbf{A}))' = (\mathbf{A}'_{1}, \mathbf{A}'_{2}, \dots,
\mathbf{A}'_{m})$. Let $\mathbf{A} \in \Re^{r \times s}$ and $\mathbf{B} \in \Re^{n \times m}$, then
$\mathbf{A} \otimes \mathbf{B} \in \Re^{sn \times rm}$ denotes its Kronecker product. Given a null matrix
$\mathbf{A} \in \Re^{n \times n}$ with diagonal elements $a_{ii} \neq 0$ for at least one $i$, this shall
be denoted by $\mathbf{A} = \diag(a_{11}, a_{22}, \dots, a_{nn})$. Given $\mathbf{a} \in \Re^{n}$, a
vector, its Euclidean norm shall be defined as $||\mathbf{a}|| = \sqrt{\mathbf{a}'\mathbf{a}} =
\sqrt{a_{1}^{2}+a_{2}^{2}+ \cdots +a_{n}^{2}}$.

If a random matrix $\mathbf{Y} \in \Re^{n \times m}$ has a matrix multivariate normal distribution with
matrix mean $\E(\mathbf{X})= \bg{\mu} \in \Re^{n \times m}$ and covariance matrix $\cov (\vec
\mathbf{Y}') = \mathbf{\Theta} \otimes \mathbf{\Sigma}$, $\mathbf{\Theta} = \mathbf{\Theta}' \in \Re^{n
\times n}$ and $\mathbf{\Sigma} = \mathbf{\Sigma}' \in \Re^{m \times m}$ this fact shall be denoted as
$\mathbf{Y} \sim \mathcal{N}_{n \times m}(\bg{\mu}, \mathbf{\Theta} \otimes \mathbf{\Sigma})$. Observe
that, if $\mathbf{A} \in \Re^{n \times r}$, $\mathbf{B} \in \Re^{m \times s}$ and $\mathbf{C} \in \Re^{r
\times s}$ matrices of constants,
\begin{equation}\label{tln}
    \mathbf{A'YB}+\mathbf{C} \sim \mathcal{N}_{r \times s}(\bg{A'\mu B}+\mathbf{C}, \mathbf{A'\Theta A}
  \otimes \mathbf{B'\Sigma B}).
\end{equation}

Finally, consider that $\mathbf{Y} \sim \mathcal{N}_{n \times m}(\bg{\mu}, \mathbf{\Theta} \otimes
\mathbf{\Sigma})$ then the random matrix $\mathbf{V} =
\mathbf{Y}'\mathbf{\Theta}^{-1}\mathbf{Y}$ has a noncentral Wishart distribution with $n$ degrees of
freedom and non-centrality  parameter matrix  $\mathbf{\Omega} =
\mathbf{\Sigma}^{-1} \bg{\mu}'\mathbf{\Theta}^{-1} \bg{\mu}/2$. This fact shall be denoted as $\mathbf{V}
\sim \mathcal{W}_{m}(n, \mathbf{\Sigma}, \mathbf{\Omega})$. Observe that if $\bg{\mu} = \mathbf{0}$, then
$\mathbf{\Omega} = \mathbf{0}$, and $\mathbf{V}$ is said to have a central Wishart distribution and
$\mathcal{W}_{m}(n, \mathbf{\Sigma}, \mathbf{0}) \equiv \mathcal{W}_{m}(n, \mathbf{\Sigma})$, see
\citet{sk:79} and \citet{mh:05}.

\subsection{General multivariate linear model}

Consider the general multivariate linear model
\begin{equation}\label{glm}
    \mathbf{Y} = \mathbf{X} \bg{\beta} + \bg{\epsilon}
\end{equation}
where: $\mathbf{Y} \in \Re^{n \times q}$ is the matrix of the observed values; $\bg{\beta} \in \Re^{p
\times q}$ is the parameter matrix; $\mathbf{X} \in \Re^{n \times p}$ is the design matrix or the
regression matrix of rank $r \leq p$ and $n > p+q$; $\bg{\epsilon} \in \Re^{n \times q}$ is the error
matrix which has a matrix multivariate normal distribution, specifically  $\bg{\epsilon}\sim
\mathcal{N}_{n \times q}(\mathbf{0}, \mathbf{I}_{n} \otimes \bg{\Sigma})$, see \citet[p.430]{mh:05} and
$\bg{\Sigma} \in \Re^{q \times q}$, $\bg{\Sigma} > \mathbf{0}$. For this model, we want to test the
hypothesis
\begin{equation}\label{test}
    H_{0}: \mathbf{C} \bg{\beta}\mathbf{M} = \mathbf{0} \mbox{ versus } H_{a}: \mathbf{C} \bg{\beta}
    \mathbf{M} \neq \mathbf{0}
\end{equation}
where $\mathbf{C} \in \Re^{\nu_{\mathbf{H}} \times p}$ of rank $\nu_{\mathbf{H}} \leq r$ and $\mathbf{M}
\in \Re^{q \times g}$ of rank $g \leq q$. As in the univariate case,  the matrix $\mathbf{C}$ concerns to
the hypothesis among the elements of the parameter matrix columns, while the matrix  $\mathbf{M}$ allows
hypothesis among the different response parameters. The matrix $\mathbf{M}$ plays a role in profile
analysis, for example; in ordinary hypothesis testing it assumes the identity matrix, namely, $\mathbf{M} =
\mathbf{I}_{p}.$

Let $\mathbf{S}_{H}$ be the matrix of sums of squares and sums of products due to the hypothesis and let
$\mathbf{S}_{E}$ be the matrix of sums of squares and sums of products due to the error. These are
defined as
\begin{equation}\label{shse}
    \begin{array}{lcl}
      \mathbf{S}_{H} &=& (\mathbf{C} \widetilde{\bg{\beta}}\mathbf{M})' (\mathbf{C}
      (\mathbf{X}' \mathbf{X})^{c} \mathbf{C}')^{-1}(\mathbf{C} \widetilde{\bg{\beta}}
      \mathbf{M}) \\
     \mathbf{S}_{E} &=& \mathbf{M}'\mathbf{Y}'(\mathbf{I}_{n} - \mathbf{X}\mathbf{X}^{c})
    \mathbf{Y}\mathbf{M},
\end{array}
\end{equation}
respectively; where $\widetilde{\bg{\beta}} = \mathbf{X}^{c}\mathbf{Y}$. Note that, under the null
hypothesis, $\mathbf{S}_{H}$ has a  $g$-dimensional noncentral Wishart distribution with
$\nu_{\mathbf{H}}$ degrees of freedom and parameter matrix $\mathbf{M}'\bg{\Sigma\mathbf{M}}$ i.e.
$\mathbf{S}_{H} \sim \mathcal{W}_{g}(\nu_{\mathbf{H}}, \mathbf{M}'\bg{\Sigma\mathbf{M}})$; similarly
$\mathbf{S}_{E}$ has a $g$-dimensional Wishart distribution with $\nu_{\mathbf{E}}$ degrees of freedom
and parameter matrix $\mathbf{M}'\bg{\Sigma\mathbf{M}}$, i.e. $\mathbf{S}_{E} \sim
\mathcal{W}_{g}(\nu_{\mathbf{E}}, \mathbf{M}'\bg{\Sigma\mathbf{M}})$; specifically, $\nu_{\mathbf{H}}$
and $\nu_{\mathbf{E}}$ denote the degrees of freedom of the hypothesis and  the error, respectively. All
the results given below are true for $\mathbf{M} \neq \mathbf{I}_{q}$, just compute $\mathbf{S}_{H}$ and
$\mathbf{S}_{E}$ from (\ref{shse}) and replace $q$ by $g$. Now, let $\lambda_{1}, \cdots, \lambda_{s}$ be
the  $s = \min(\nu_{\mathbf{H}}, g)$ non null eigenvalues of the matrix
$\mathbf{S}_{H}\mathbf{S}_{E}^{-1}$ such that $0 < \lambda_{s}< \cdots < \lambda_{1} < \infty$ and let
$\theta_{1}, \cdots, \theta_{s}$ be the $s$ non null eigenvalues of the matrix
$\mathbf{S}_{H}(\mathbf{S}_{H} + \mathbf{S}_{E})^{-1}$ with $0 < \theta_{s}< \cdots < \theta_{1} < 1$;
here we note $\lambda_{i}= \theta_{i}/(1 - \theta_{i})$ and $\theta_{i}= \lambda_{i}/(1 + \lambda_{i})$,
$i = 1,\cdots, s$. Various authors have proposed a number of different criteria for testing the
hypothesis (\ref{test}). But it is known,  that all the tests can be expressed in terms of the
eigenvalues $\lambda's$ or $\theta's$, see for example \citet{k:83}. The  likelihood ratio test
statistics termed Wilks's $\Lambda$, given next, is one of such statistic.

The likelihood ratio test of size $\alpha$ of $H_{0}: \mathbf{C} \bg{\beta}\mathbf{M} = \mathbf{0} \mbox{
against } H_{a}: \mathbf{C} \bg{\beta}  \mathbf{M} \neq \mathbf{0}$ reject if $\Lambda \leq
\Lambda_{\alpha,q,\nu_{\mathbf{H}},\nu_{\mathbf{E}}}$, where
\begin{equation}\label{w}
    \Lambda = \frac{\left|\mathbf{S}_{\mathbf{E}}\right|}{\left|\mathbf{S}_{\mathbf{E}} + \mathbf{S}_{\mathbf{H}}\right|} =
  \prod_{i = 1}^{s}\frac{1}{\left(1+\lambda_{i}\right)} = \prod_{i = 1}^{s}\frac{1}{\left(1-\theta_{i}\right)}.
\end{equation}
Exact critical values of $\Lambda_{\alpha,q,\nu_{\mathbf{H}},\nu_{\mathbf{E}}}$ can be found in
\citet[Table A.9]{re:95} or \citet[Table 1]{k:83}.

\section{Test about $R$ multivariate simple linear models}\label{sec:3}

Consider the following $R$ multivariate simple linear models
\begin{equation}\label{gmslm}
    \mathbf{Y}_{r} = \mathbf{X}_{r} \bg{B}_{r} + \bg{\epsilon}_{r}
\end{equation}
$\mathbf{Y}_{r} \in \Re^{n_{r} \times q}$, $\mathbf{X}_{r}  \in \Re^{n_{r} \times 2}$ such that its rank
is 2; $\bg{B}_{r} \in \Re^{2 \times q}$, $n_{r} > 2$ and $n_{r} > q +2$ for all $r = 1,2,\dots,R$;
$\sum_{r=1}^{R} n_{r}=N$ and $\bg{\epsilon}_{r} \sim \mathcal{N}_{n_{r} \times q}(\mathbf{0},
\mathbf{I}_{n_{r}} \otimes \mathbf{\Sigma})$, with $\mathbf{\Sigma} > \mathbf{0}$; where
$$
  \bg{B}_{r} =
  \left (
     \begin{array}{cccc}
       \alpha_{r1} & \alpha_{r2} & \cdots & \alpha_{rq} \\
       \beta_{r1} & \beta_{r2} & \cdots & \beta_{rq1}
     \end{array}
  \right)=
  \left (
     \begin{array}{c}
       \bg{\alpha}'_{r} \\
       \bg{\beta}'_{r}
     \end{array}
  \right), \quad
  \mathbf{X_{r}} =
  \left (
     \begin{array}{cc}
       1 & x_{r1} \\
       1 & x_{r2} \\
       \vdots & \vdots \\
       1 & x_{rn_{r}}
     \end{array}
  \right ) = \left (\mathbf{1}_{n_{r}} \mathbf{x}_{r} \right ).
$$

We are interested in the following hypotheses
\begin{description}
  \item[i)] $H_{0}: \bg{\beta}_{1} = \bg{\beta}_{2} = \cdots = \bg{\beta}_{R}$, that is, the set of lines
  are parallel;
  \item[ii)] $H_{0}: \bg{\alpha}_{1} = \bg{\alpha}_{2} = \cdots = \bg{\alpha}_{R}$, that is, the set of lines
  have a common vector intercept;
  \item[iii)] $H_{0}: \bg{\alpha}_{1} + \bg{\beta}_{1}x_{0} = \bg{\alpha}_{2} + \bg{\beta}_{2}x_{0} = \cdots =
  \bg{\alpha}_{R} + \bg{\beta}_{R}x_{0}$, ($x_{0}$ known), that is, the set of lines intersect at the $x$ value
   $x_{0}$ which is specified in advance.
\end{description}

First observe that the $R$ multivariate simple linear models can be written as a general multivariate
linear model, $\mathbb{Y} = \mathbb{XB} + \mathbb{E}$, such that {\small
$$
  \build{\mathbb{Y}}{}{N \times q} =
  \left (
     \begin{array}{c}
       \mathbf{Y}_{1} \\
       \mathbf{Y}_{2} \\
       \vdots \\
       \mathbf{Y}_{R}
     \end{array}
  \right ),
  \build{\mathbb{X}}{}{N \times 2R} =
  \left(
     \begin{array}{cccc}
       \mathbf{X}_{1} & \mathbf{0} & \cdots & \mathbf{0} \\
       \mathbf{0} & \mathbf{X}_{2} & \cdots & \mathbf{0} \\
       \vdots & \vdots & \ddots & \vdots \\
       \mathbf{0} & \mathbf{0} & \cdots & \mathbf{X}_{R}
     \end{array}
  \right ),
  \build{\mathbb{B}}{}{2R \times q} =
  \left (
     \begin{array}{c}
       \mathbf{B}_{1} \\
       \mathbf{B}_{2} \\
       \vdots \\
       \mathbf{B}_{R}
     \end{array}
  \right ),
  \build{\mathbb{E}}{}{N \times q} =
  \left (
     \begin{array}{c}
       \mathbf{E}_{1} \\
       \mathbf{E}_{2} \\
       \vdots \\
       \mathbf{E}_{R}
     \end{array}
  \right ).
$$}
Namely, $\mathbb{E} \sim \mathcal{N}_{N \times q}(\mathbf{0}, \mathbf{I}_{NR} \otimes
\mathbf{\Sigma})$. Thus
$$
  \mathbb{X}'\mathbb{X} =
  \left(
     \begin{array}{cccc}
       \mathbf{X}'_{1}\mathbf{X}_{1} & \mathbf{0} & \cdots & \mathbf{0} \\
       \mathbf{0} & \mathbf{X}'_{2}\mathbf{X}_{2} & \cdots & \mathbf{0} \\
       \vdots & \vdots & \ddots & \vdots \\
       \mathbf{0} & \mathbf{0} & \cdots & \mathbf{X}'_{R}\mathbf{X}_{R}
     \end{array}
  \right ),
  \quad
  \mathbb{X}'\mathbb{Y} =
  \left (
     \begin{array}{c}
       \mathbf{X}'_{1}\mathbf{Y}_{1} \\
       \mathbf{X}'_{2}\mathbf{Y}_{2} \\
       \vdots \\
       \mathbf{X}'_{R}\mathbf{Y}_{R}
     \end{array}
  \right ),
$$
and by \citet[Theorem 1.3.1, p. 19]{g:76}
$$
  (\mathbb{X}'\mathbb{X})^{-1} =
  \left(
     \begin{array}{cccc}
       (\mathbf{X}'_{1}\mathbf{X}_{1})^{-1} & \mathbf{0} & \cdots & \mathbf{0} \\
       \mathbf{0} & (\mathbf{X}'_{2}\mathbf{X}_{2})^{-1} & \cdots & \mathbf{0} \\
       \vdots & \vdots & \ddots & \vdots \\
       \mathbf{0} & \mathbf{0} & \cdots & (\mathbf{X}'_{R}\mathbf{X}_{R})^{-1}
     \end{array}
  \right ).
 $$
Therefore by \citet[Theorem 10.1.1, p. 430]{mh:05}, see also \citet[equation 10.46, p. 339]{re:95},
$$
  \widehat{\mathbb{B}} = (\mathbb{X}'\mathbb{X})^{-1}\mathbb{X}'\mathbb{Y} =
  \left(
     \begin{array}{cccc}
       (\mathbf{X}'_{1}\mathbf{X}_{1})^{-1} & \mathbf{0} & \cdots & \mathbf{0} \\
       \mathbf{0} & (\mathbf{X}'_{2}\mathbf{X}_{2})^{-1} & \cdots & \mathbf{0} \\
       \vdots & \vdots & \ddots & \vdots \\
       \mathbf{0} & \mathbf{0} & \cdots & (\mathbf{X}'_{R}\mathbf{X}_{R})^{-1}
     \end{array}
  \right )
  \left (
     \begin{array}{c}
       \mathbf{X}'_{1}\mathbf{Y}_{1} \\
       \mathbf{X}'_{2}\mathbf{Y}_{2} \\
       \vdots \\
       \mathbf{X}'_{R}\mathbf{Y}_{R}
     \end{array}
  \right )
$$
$$
  =
  \left (
     \begin{array}{c}
       (\mathbf{X}'_{1}\mathbf{X}_{1})^{-1}\mathbf{X}'_{1}\mathbf{Y}_{1} \\
       (\mathbf{X}'_{2}\mathbf{X}_{2})^{-1}\mathbf{X}'_{2}\mathbf{Y}_{2} \\
       \vdots \\
       (\mathbf{X}'_{R}\mathbf{X}_{R})^{-1}\mathbf{X}'_{R}\mathbf{Y}_{R}
     \end{array}
  \right ) =
  \left (
     \begin{array}{c}
       \mathbf{X}_{1}^{-}\mathbf{Y}_{1} \\
       \mathbf{X}_{2}^{-}\mathbf{Y}_{2} \\
       \vdots \\
       \mathbf{X}_{R}^{-}\mathbf{Y}_{R}
     \end{array}
  \right ),
$$
that is, $\mathbf{B}_{r} = (\mathbf{X}'_{r}\mathbf{X}_{r})^{-1}\mathbf{X}'_{r}\mathbf{Y}_{r} =
\mathbf{X}_{r}^{-}\mathbf{Y}_{r}$ and
\begin{eqnarray}
  \mathbf{S}_{\mathbf{E}} &=& (\mathbb{Y} - \widehat{\mathbb{B}}'\mathbb{X}'\mathbb{Y})'(\mathbb{Y} -
  \widehat{\mathbb{B}}'\mathbb{X}'\mathbb{Y}) = \sum_{r = 1}^{R}\mathbf{Y}_{r}\mathbf{Y}_{r} - \sum_{r = 1}^{R}
  \mathbf{B}'_{r}\mathbf{X}'_{r}\mathbf{Y}_{r} \nonumber \\ \label{SE}
    &=& \sum_{r = 1}^{R}\mathbf{Y}'_{r}\left(\mathbf{I}_{n_{r}}-\mathbf{X}_{r}\mathbf{X}_{r}^{-}\right)
    \mathbf{Y}_{r} \in \Re^{q \times q}.
\end{eqnarray}
Hence by \citet[Theorem 10.1.2, p. 431]{mh:05} and \citet[equation 6.3.8, p. 171]{sk:79} we have that
$\widehat{\mathbb{B}} \sim \mathcal{N}_{2R \times q}\left(\mathbb{B},
\left(\mathbb{X}'\mathbb{X}\right)^{-1} \otimes \mathbf{\Sigma}\right)$.

Note that
$$
  \widehat{\mathbf{B}}_{r} = \left(\build{\mathbf{0}}{}{1} \build{\cdots}{}{\cdots} \build{\mathbf{I}_{2}}{}{r} \build{\cdots}{}{\cdots}
  \build{\mathbf{0}}{}{R}\right)
  \widehat{\mathbb{B}} =
  \left(\build{\mathbf{0}}{}{1} \build{\cdots}{}{\cdots} \build{\mathbf{I}_{2}}{}{r} \build{\cdots}{}{\cdots}
  \build{\mathbf{0}}{}{R}\right)
  \left (
     \begin{array}{c}
       \widehat{\mathbf{B}}_{1} \\
       \widehat{\mathbf{B}}_{2} \\
       \vdots \\
       \widehat{\mathbf{B}}_{R}
     \end{array}
  \right ),
$$
thus, $\widehat{\mathbf{B}}_{r} \sim \mathcal{N}_{2 \times q}\left(\mathbf{B}_{r},
\left(\mathbf{X}'_{r}\mathbf{X}_{r}\right)^{-1} \otimes \mathbf{\Sigma}\right)$ and
$\mathbf{S}_{\mathbf{E}} \sim \mathcal{W}_{q}(N-2R, \mathbf{\Sigma})$. Observe that
$\widehat{\mathbf{B}}_{r}$ is computed from the data for the $r$th model and $\mathbf{S}_{\mathbf{E}}$ is
computed by pooling the estimators of $\mathbf{S}_{\mathbf{E}}$ from each model
$\mathbf{S}_{\mathbf{E}_{r}}$.

Generalising the results in \citet[Example 6.2.1, pp. 177-178]{g:76} and using matrix notation in the
multivariate case, we have
$$
  \widehat{\mathbf{B}}_{r} =
  \left (
     \begin{array}{c}
       \widehat{\bg{\alpha}}'_{r} \\
       \widehat{\bg{\beta}}'_{r}
     \end{array}
  \right ) =
  \left (
     \begin{array}{c}
       \left (\bar{\mathbf{Y}}_{r} - \widehat{\bg{\beta}}_{r} \bar{x}_{r}\right )' \\
       \left (\displaystyle\frac{\mathbf{Y}'_{r}\left (\mathbf{I}_{n_{r}} - \mathbf{1}_{n_{r}} \mathbf{1}'_{n_{r}}/n_{r}\right )
       \mathbf{x}_{r}}{\|\left (\mathbf{I}_{n_{r}} - \mathbf{1}_{n_{r}} \mathbf{1}'_{n_{r}}/n_{r}\right )
       \mathbf{x}_{r}\|^{2}} \right )'
     \end{array}
  \right ),
$$
where $\bar{\mathbf{Y}}_{r} = \mathbf{Y}'_{r}\mathbf{1}_{n_{r}}/n_{r}$ and $\bar{x}_{r} =
\mathbf{x}'_{r}\mathbf{1}_{n_{r}}/n_{r}$, $r = 1, 2 , \dots, R$. And
$$
  \mathbf{S}_{\mathbf{E}} = \sum_{r = 1}^{R}\mathbf{S}_{\mathbf{E}_{r}},
$$
where
$$
  \mathbf{S}_{\mathbf{E}_{r}} = \mathbf{Y}'_{r}\left (\mathbf{I}_{n_{r}} - \mathbf{1}_{n_{r}} \mathbf{1}'_{n_{r}}/n_{r}\right )
       \mathbf{Y}_{r} \hspace{7cm}
$$
$$
\hspace{3.5cm}
  - \frac{\mathbf{Y}'_{r}\left (\mathbf{I}_{n_{r}} - \mathbf{1}_{n_{r}} \mathbf{1}'_{n_{r}}/n_{r}\right )
       \mathbf{x}_{r} \mathbf{x}'_{r}\left (\mathbf{I}_{n_{r}} - \mathbf{1}_{n_{r}} \mathbf{1}'_{n_{r}}/n_{r}\right )
       \mathbf{Y}_{r}}{\|\left (\mathbf{I}_{n_{r}} - \mathbf{1}_{n_{r}} \mathbf{1}'_{n_{r}}/n_{r}\right )
       \mathbf{x}_{r}\|^{2}}.
$$
\begin{theorem}\label{teo1}
Given the $R$ multivariate simple linear models (\ref{gmslm}) and known constants $a$ and $b$, the
likelihood ratio test of size $\alpha$ of
$$
  H_{0}: a\bg{\alpha}_{1} + b\bg{\beta}_{1} = a\bg{\alpha}_{2} + b\bg{\beta}_{2} = \cdots =
  a\bg{\alpha}_{R} + b\bg{\beta}_{R}
$$
versus
$$
  H_{1}: \mbox{ at least ane equality is an inequality,}
$$
is given by
\begin{equation}\label{wilks}
    \Lambda = \frac{|\mathbf{S}_{E}|}{|\mathbf{S}_{E} + \mathbf{S}_{H}|}
\end{equation}
Where
\begin{eqnarray}\label{se}
  \mathbf{S}_{E} &=& \sum_{r = 1}^{R}\mathbf{Y}'_{r}(\mathbf{I}_{n_{r}} - \mathbf{X}_{r}
        \mathbf{X}^{-}_{r}) \mathbf{Y}_{r},\\
  \label{sh}
  \mathbf{S}_{H} &=& \left(\mathbf{D}^{1/2}\mathbf{Z}\right )'\left(\mathbf{I}_{R} -
  \mathbf{D}^{1/2}\mathbf{1}_{R}\mathbf{1}'_{R}\mathbf{D}^{1/2}/\mathbf{1}'_{R}\mathbf{D}\mathbf{1}_{R}\right)
  \left(\mathbf{D}^{1/2}\mathbf{Z}\right ),
\end{eqnarray}
where $\mathbf{D} = \diag(d_{11}, d_{22}, \dots, d_{RR})$,
$$
  d_{rr} = \frac{n_{r} \|\left (\mathbf{I}_{n_{r}} - \mathbf{1}_{n_{r}} \mathbf{1}'_{n_{r}}/n_{r}\right )
       \mathbf{x}_{r}\|^{2}}{\|(a\mathbf{x}_{r}-b\mathbf{1}_{n_{r}})\|^{2}}
$$
and
$$
  \mathbf{Z} = \left (a
  \left (
     \begin{array}{c}
       \widehat{\bg{\alpha}}'_{1} \\
       \widehat{\bg{\alpha}}'_{2} \\
       \vdots \\
       \widehat{\bg{\alpha}}'_{R}
     \end{array}
  \right ) +  b
  \left (
     \begin{array}{c}
       \widehat{\bg{\beta}}'_{1} \\
       \widehat{\bg{\beta}}'_{2} \\
       \vdots \\
       \widehat{\bg{\beta}}'_{R}
     \end{array}
  \right )
  \right ) \in \Re^{R \times q}.
$$

We reject $H_{0}$ if
$$
  \Lambda \leq \Lambda_{\alpha, 1, \nu_{\mathbf{H}}, \nu_{\mathbf{E}}},
$$
where $\nu_{\mathbf{H}} = (R-1)$, $\nu_{\mathbf{E}} = N-2R$.
\end{theorem}
\begin{proof}
This theorem is a special case of the results obtained for testing the hypotheses (\ref{test}) and it can be
proved by selecting the proper $\mathbf{C}$ and $\mathbf{M}$ matrices into Equation (\ref{shse})
\footnote{In our case taking, $\mathbf{M} = \mathbf{I}_{q}$ and
$$
  \mathbf{C} =
  \left (
     \begin{array}{ccccccccccc}
       a & b & -a & -b & 0 & 0 & \cdots & 0 & 0 & 0 & 0\\
       0 & 0 & a & b & -a & -b & \cdots & 0 & 0 & 0 & 0\\
       \vdots & \vdots & \vdots & \vdots & \vdots & \vdots & \ddots & \vdots & \vdots & \vdots & \vdots\\
       0 & 0 & 0 & 0 & 0 & 0 & \cdots & a & b & -a & -b
     \end{array}
  \right ) \in \Re^{R-1 \times 2R},
$$
into to Equation \ref{shse}, the desired result is obtained.}. Alternatively we extend the proof in
\citet[Theorem 8.6.1, p. 288]{g:76} for an univariate case into the multivariate case. The result follows from
(\ref{w}), we just need to  define explicit matrices of sums of squares and products $\mathbf{S}_{E}$
and $\mathbf{S}_{H}$. First  define the random vectors $\mathbf{z}_{r} = a
\widehat{\bg{\alpha}}_{r} + b \widehat{\bg{\beta}}_{r}$, $r = 1, 2, \dots, R$, where $a$ and $b$ are
known constants to be define later. Hence, given that $\widehat{\mathbf{B}}_{r} \sim \mathcal{N}_{2
\times q}\left(\mathbf{B}_{r}, \left(\mathbf{X}'_{r}\mathbf{X}_{r}\right)^{-1} \otimes
\mathbf{\Sigma}\right)$, we have
$$
  \E(\mathbf{z}_{r}) = \E\left(a \widehat{\bg{\alpha}}_{r} + b \widehat{\bg{\beta}}_{r}\right ) =
   a \bg{\alpha}_{r} + b \bg{\beta}_{r}.
$$
Also note that,
$$
  \mathbf{z}_{r} = \widehat{\mathbf{B}}'_{r}
  \left (
     \begin{array}{c}
       a \\
       b
     \end{array}
  \right ) = a \widehat{\bg{\alpha}}_{r} + b \widehat{\bg{\beta}}_{r},
$$
thus
\begin{eqnarray*}
  \cov(\mathbf{z}_{r}) &=& \cov(\mathbf{z}_{r}) = \cov\left(\vec \widehat{\mathbf{B}}'_{r}
  \left (
     \begin{array}{c}
       a \\
       b
     \end{array}
  \right )\right )\\
  &=& \cov\left(\left(
  \left (
     \begin{array}{c}
       a \\
       b
     \end{array}
  \right )' \otimes \mathbf{I}_{q} \right ) \vec \widehat{\mathbf{B}}'_{r}
    \right ) \\
   &=& \left((a, b) \otimes \mathbf{I}_{q} \right ) \left( \left(\mathbf{X}'_{r}\mathbf{X}_{r}\right )^{-1}
   \otimes \mathbf{\Sigma}\right ) \left(
   \left (
     \begin{array}{c}
       a \\
       b
     \end{array}
  \right ) \otimes  \mathbf{I}_{q} \right ) \\
   &=& (a, b)\left(\mathbf{X}'_{r}\mathbf{X}_{r}\right )^{-1}
   \left (
     \begin{array}{c}
       a \\
       b
     \end{array}
  \right )  \otimes \mathbf{\Sigma}\\
   &=& d_{rr}^{-1}\otimes \mathbf{\Sigma} = d_{rr}^{-1}\mathbf{\Sigma}
\end{eqnarray*}
With
\begin{eqnarray*}
  d_{rr}^{-1} &=& (a, b)\left(\mathbf{X}'_{r}\mathbf{X}_{r}\right )^{-1}
   \left (
     \begin{array}{c}
       a \\
       b
     \end{array}
  \right ) \\
   &=& (a, b)\left(
     \begin{array}{cc}
       \|\mathbf{1}_{nr}\|^{2} & \mathbf{1}'_{nr}\mathbf{x}_{r} \\
       \mathbf{x}'_{r}\mathbf{1}_{nr} & \|\mathbf{x}_{r}\|^{2}
     \end{array}
  \right )^{-1}
   \left (
     \begin{array}{c}
       a \\
       b
     \end{array}
  \right ) \\
   &=& \frac{1}{n_{r} \|\left (\mathbf{I}_{n_{r}} - \mathbf{1}_{n_{r}} \mathbf{1}'_{n_{r}}/n_{r}\right )
       \mathbf{x}_{r}\|^{2}}(a, b)\left(
     \begin{array}{cc}
       \|\mathbf{x}_{r}\|^{2} & -\mathbf{1}'_{nr}\mathbf{x}_{r} \\
       -\mathbf{x}'_{r}\mathbf{1}_{nr} & n_{r}
     \end{array}
  \right )^{-1}
   \left (
     \begin{array}{c}
       a \\
       b
     \end{array}
  \right )\\
  &=& \frac{\|'(a\mathbf{x}_{r}-b\mathbf{1}_{n_{r}})\|^{2}}
      {n_{r} \|\left (\mathbf{I}_{n_{r}} - \mathbf{1}_{n_{r}} \mathbf{1}'_{n_{r}}/n_{r}\right )
      \mathbf{x}_{r}\|^{2}},
\end{eqnarray*}
Therefore
$$
  \mathbf{z}_{r}  = a \widehat{\bg{\alpha}}_{r} + b \widehat{\bg{\beta}}_{r} \sim \mathcal{N}_{q} \left (
  a \bg{\alpha}_{r} + b \bg{\beta}_{r}, d_{rr}^{-1} \mathbf{\Sigma}\right ).
$$
Now, consider the random matrix $\mathbf{Z}$ defined by
$$
  \mathbf{Z} =
  \left (
     \begin{array}{c}
       \mathbf{z}'_{1} \\
       \mathbf{z}'_{2} \\
       \vdots \\
       \mathbf{z}'_{R}
     \end{array}
  \right ) =
  \left (a
  \left (
     \begin{array}{c}
       \widehat{\bg{\alpha}}'_{1} \\
       \widehat{\bg{\alpha}}'_{2} \\
       \vdots \\
       \widehat{\bg{\alpha}}'_{R}
     \end{array}
  \right ) +  b
  \left (
     \begin{array}{c}
       \widehat{\bg{\beta}}'_{1} \\
       \widehat{\bg{\beta}}'_{2} \\
       \vdots \\
       \widehat{\bg{\beta}}'_{R}
     \end{array}
  \right )
  \right )\in \Re^{R \times q}
$$
Thus
$$
  \E(\mathbf{Z}) = \left (a
  \left (
     \begin{array}{c}
       \bg{\alpha}'_{1} \\
       \bg{\alpha}'_{2} \\
       \vdots \\
       \bg{\alpha}'_{R}
     \end{array}
  \right ) +  b
  \left (
     \begin{array}{c}
       \bg{\beta}'_{1} \\
       \bg{\beta}'_{2} \\
       \vdots \\
       \bg{\beta}'_{R}
     \end{array}
  \right )
  \right )
$$
and
$$
  \cov(\vec \mathbf{Z}') = \cov((\mathbf{z}'_{1}, \mathbf{z}'_{2}, \dots \mathbf{z}'_{R})') = \mathbf{D}^{-1}
  \otimes \mathbf{\Sigma},
$$
where $\mathbf{D} = \diag(d_{11}, d_{22}, \dots,d_{RR})$. Thus
$$
  \mathbf{Z} \sim \mathcal{N}_{R \times q} \left (a
  \left (
     \begin{array}{c}
       \widehat{\bg{\alpha}}'_{1} \\
       \widehat{\bg{\alpha}}'_{2} \\
       \vdots \\
       \widehat{\bg{\alpha}}'_{R}
     \end{array}
  \right ) +  b
  \left (
     \begin{array}{c}
       \widehat{\bg{\beta}}'_{1} \\
       \widehat{\bg{\beta}}'_{2} \\
       \vdots \\
       \widehat{\bg{\beta}}'_{R}
     \end{array}
  \right ), \mathbf{D}^{-1} \otimes \mathbf{\Sigma} \right ),
$$
furthermore
$$
  \mathbf{D}^{1/2}\mathbf{Z} \sim \mathcal{N}_{R \times q} \left (\mathbf{D}^{1/2} \left (a
  \left (
     \begin{array}{c}
       \widehat{\bg{\alpha}}'_{1} \\
       \widehat{\bg{\alpha}}'_{2} \\
       \vdots \\
       \widehat{\bg{\alpha}}'_{R}
     \end{array}
  \right ) +  b
  \left (
     \begin{array}{c}
       \widehat{\bg{\beta}}'_{1} \\
       \widehat{\bg{\beta}}'_{2} \\
       \vdots \\
       \widehat{\bg{\beta}}'_{R}
     \end{array}
  \right ) \right ), \mathbf{I}_{R} \otimes \mathbf{\Sigma} \right ).
$$
Consider the constant matrix $\left (\mathbf{I}_{R}- \mathbf{D}^{1/2}\mathbf{1}_{R} \mathbf{1}'_{R}
\mathbf{D}^{1/2} /\mathbf{1}'_{R}\mathbf{D}\mathbf{1}_{R}\right )$, which is symmetric and idempotent.
Then
$$
  \mathbf{S}_{\mathbf{H}} = \left(\mathbf{D}^{1/2} \mathbf{Z}\right)'\left (\mathbf{I}_{R}- \mathbf{D}^{1/2}\mathbf{1}_{R} \mathbf{1}'_{R}
\mathbf{D}^{1/2} /\mathbf{1}'_{R}\mathbf{D}\mathbf{1}_{R}\right ) \left(\mathbf{D}^{1/2}
\mathbf{Z}\right),
$$
moreover, $\mathbf{S}_{\mathbf{H}}$ has a Wishart distribution and is independently distributed of
$\mathbf{S}_{\mathbf{E}}$ (see Equation (\ref{SE})), where  $\mathbf{S}_{\mathbf{H}} \sim
\mathcal{W}_{q}(R-1,\mathbf{\Sigma}, \mathbf{\Omega})$ and $\mathbf{S}_{\mathbf{E}} \sim
\mathcal{W}_{q}(N-2R,\mathbf{\Sigma})$; in addition,
$$
  \mathbf{\Omega} = \frac{1}{2}\mathbf{\Sigma}^{-1}\left(\mathbf{D}^{1/2}
  \E(\mathbf{Z}) \right)'
  \left (\mathbf{I}_{R}- \mathbf{D}^{1/2}\mathbf{1}_{R} \mathbf{1}'_{R}
  \mathbf{D}^{1/2} /\mathbf{1}'_{R}\mathbf{D}\mathbf{1}_{R}\right )
  \left(\mathbf{D}^{1/2}
  \E(\mathbf{Z}) \right)
$$
and observe that $\mathbf{\Omega}= \mathbf{0}$ if an only if $a\bg{\alpha}_{1} + b\bg{\beta}_{1} =
a\bg{\alpha}_{2} + b\bg{\beta}_{2} = \cdots = a\bg{\alpha}_{R} + b\bg{\beta}_{R}$. Which is the
desired result.
\end{proof}

As we mentioned before, different test statistics have been proposed for verifying the hypothesis
(\ref{test}). Next we propose three of them in our particular case.

\begin{theorem}\label{teo2}
Given the $R$ multivariate simple linear models (\ref{gmslm}) and known constants $a$ and $b$, the
union-intersection test, Pillai test and Lawley-Hotelling test of size $\alpha$ of
$$
  H_{0}: a\bg{\alpha}_{1} + b\bg{\beta}_{1} = a\bg{\alpha}_{2} + b\bg{\beta}_{2} = \cdots =
  a\bg{\alpha}_{R} + b\bg{\beta}_{R}
$$
versus
$$
  H_{1}: \mbox{ at least ane equality is an inequality,}
$$
are given respectively by
\begin{enumerate}
  \item \begin{equation}\label{roy}
       \theta_{1} = \frac{\lambda_{1}}{1 + \lambda_{1}}
    \end{equation}
    which is termed Roy's largest root test. Where $\lambda_{1}$ is the maximum eigenvalue of $\left
    (\mathbf{S}_{H}\mathbf{S}_{E}^{-1}\right )$, where $\mathbf{S}_{H}$ and $\mathbf{S}_{E}$ are given by
    (\ref{sh}) and (\ref{se}), respectively. We reject $H_{0}$ if $\theta \geq \theta_{\alpha, s,m,h}$. Exact critical
    values of $\theta_{\alpha, s,m,h}$ are found in \citet[Table A.10]{re:95} or \citet[Tables 2, 4 and
    5]{k:83}.
  \item \begin{equation}\label{Pillai}
            V^{(s)} = \tr[\mathbf{S}_{H}(\mathbf{S}_{E} + \mathbf{S}_{H})^{-1}] =
            \sum_{i = 1}^{s}\frac{\lambda_{i}}{1+\lambda_{i}} = \sum_{i = 1}^{s}
            \theta_{i}
        \end{equation}
        This way we reject $H_{0}$ if
        $$
          V^{(s)} \geq V^{(s)}_{\alpha, s,m,h},
        $$
        where the exact critical values of $V^{(s)}_{\alpha, s,m,h}$ are found in
        \citet[Table A.11]{re:95} or \citet[Table 7]{k:83}.
  \item \begin{equation}\label{L-H}
            U^{(s)} = \tr[\mathbf{S}_{H}\mathbf{S}_{E}^{-1}] = \sum_{i = 1}^{s}
            \lambda_{i} = \sum_{i = 1}^{s}\frac{\theta_{i}}{1-\theta_{i}}.
        \end{equation}
        We reject $H_{0}$ if
        $$
          U^{(s)} \geq U^{(s)}_{\alpha, s,m,h}.
        $$
        The upper percentage points, $U^{(s)}_{\alpha, s,m,h}$, are given in
        \citet[Table 6]{k:83}.
\end{enumerate}
The parameters $s$, $m$ and $h$ are defined as
$$
  s = \min(1, \nu_{\mathbf{H}}), \ \ m = (|1-\nu_{\mathbf{H}}|- 1)/2, \ \ h = (\nu_{\mathbf{E}}-2)/2.
$$
where $\nu_{\mathbf{H}} = (R-1)$, $\nu_{\mathbf{E}} = N-2R$ and $N = \sum_{r = 1}^{R} n_{r}$.
\end{theorem}

As a special case of Theorem \ref{teo1} (and Theorem \ref{teo2}),  we obtain the test of the hypotheses
\textbf{i}), \textbf{ii}) and \textbf{iii}) established above.

\begin{theorem}\label{teo3}
Consider the $R$ multivariate simple linear models (\ref{gmslm}). The likelihood ratio test of size
$\alpha$ of tests of hypotheses \textbf{i}), \textbf{ii}) and \textbf{iii})  are given as follows:

The test of $H_{0}$ vs. $H_{1}$ is this: Reject $H_{0}$ if and only if
  \begin{equation*}\label{wilks11}
    \Lambda = \frac{|\mathbf{S}_{E}|}{|\mathbf{S}_{E} + \mathbf{S}_{H}|} \leq \Lambda_{\alpha, 1, \nu_{\mathbf{H}},
    \nu_{\mathbf{E}}},
  \end{equation*}
  where
  \begin{eqnarray}\label{se11}
    \mathbf{S}_{E} &=& \sum_{r = 1}^{R}\mathbf{Y}'_{r}(\mathbf{I}_{n_{r}} - \mathbf{X}_{r}
        \mathbf{X}^{-}_{r}) \mathbf{Y}_{r},\\
     \label{sh11}
    \mathbf{S}_{H} &=& \left(\mathbf{D}^{1/2}\mathbf{Z}\right )'\left(\mathbf{I}_{R} -
    \mathbf{D}^{1/2}\mathbf{1}_{R}\mathbf{1}'_{R}\mathbf{D}^{1/2}/\mathbf{1}'_{R}\mathbf{D}\mathbf{1}_{R}\right)
    \left(\mathbf{D}^{1/2}\mathbf{Z}\right ),
  \end{eqnarray}
\begin{description}
  \item[i)] With $H_{0}: \bg{\alpha}_{1} = \bg{\alpha}_{2} = \cdots = \bg{\alpha}_{R}$ ($R$ set of lines
  with the same vector intercept) vs. $H_{1}: \bg{\alpha}_{i} = \bg{\alpha}_{j}$ for at least  one
  $i \neq j$, $i,j = 1,2,\dots,R$. Where $\mathbf{D} = \diag(d_{11}, d_{22}, \dots, d_{RR})$,
  $$
    d_{rr} = \frac{n_{r} \|\left (\mathbf{I}_{n_{r}} - \mathbf{1}_{n_{r}} \mathbf{1}'_{n_{r}}/n_{r}\right )
       \mathbf{x}_{r}\|^{2}}{\|\mathbf{x}_{r}\|^{2}}
  $$
  and
  $$
    \mathbf{Z} =
    \left (
       \begin{array}{c}
         \widehat{\bg{\alpha}}'_{1} \\
         \widehat{\bg{\alpha}}'_{2} \\
         \vdots \\
         \widehat{\bg{\alpha}}'_{R}
       \end{array}
    \right ) \in \Re^{R \times q}.
  $$
  \item[ii)] $H_{0}: \bg{\beta}_{1} = \bg{\beta}_{2} = \cdots = \bg{\beta}_{R}$ ($R$ set of lines
  are parallel) vs. $H_{1}: \bg{\beta}_{i} = \bg{\beta}_{j}$ for at least  one
  $i \neq j$, $i,j = 1,2,\dots,R$. With $\mathbf{D} = \diag(d_{11}, d_{22}, \dots, d_{RR})$,
  $$
    d_{rr} = \|\left (\mathbf{I}_{n_{r}} - \mathbf{1}_{n_{r}} \mathbf{1}'_{n_{r}}/n_{r}\right )
       \mathbf{x}_{r}\|^{2}
  $$
  and
  $$
  \mathbf{Z} =
  \left (
     \begin{array}{c}
       \widehat{\bg{\beta}}'_{1} \\
       \widehat{\bg{\beta}}'_{2} \\
       \vdots \\
       \widehat{\bg{\beta}}'_{R}
     \end{array}
    \right ) \in \Re^{R \times q}.
  $$
  \item[iii)] $H_{0}: \bg{\alpha}_{1} + \bg{\beta}_{1}x_{0} = \bg{\alpha}_{2} + ¡\bg{\beta}_{2}x_{0} = \cdots =
  \bg{\alpha}_{R} + \bg{\beta}_{R}x_{0}$ (all $R$ set of lines intersect at $x = x_{0}$, known) vs. $H_{1}$
  at least one equality is an inequality (all $R$ set of lines do not intersect at $x = x_{0}$). Where
  $\mathbf{D} = \diag(d_{11}, d_{22}, \dots, d_{RR})$,
  $$
    d_{rr} = \frac{n_{r} \|\left (\mathbf{I}_{n_{r}} - \mathbf{1}_{n_{r}} \mathbf{1}'_{n_{r}}/n_{r}\right )
       \mathbf{x}_{r}\|^{2}}{\|(\mathbf{x}_{r}-x_{0}\mathbf{1}_{n_{r}})\|^{2}}
  $$
  and
  $$
    \mathbf{Z} =
    \left (
       \begin{array}{c}
         \widehat{\bg{\alpha}}'_{1} + \widehat{\bg{\beta}}'_{1}x_{0}\\
         \widehat{\bg{\alpha}}'_{2} + \widehat{\bg{\beta}}'_{2}x_{0}\\
         \vdots \\
         \widehat{\bg{\alpha}}'_{R}+ \widehat{\bg{\beta}}'_{R}x_{0}
       \end{array}
    \right ) \in \Re^{R \times q}.
$$
\end{description}
Where $\nu_{\mathbf{H}} = (R-1)$, $\nu_{\mathbf{E}} = N-2R$.
\end{theorem}
\begin{proof}
  This is a simple consequence of Theorem \ref{teo1}. To test that a set of $R$ lines have the same vector intercept, take $a
  =1$ and $b=0$; to test whether set of $R$  lines are parallel, we set $a = 0$ and $b=1$, and to test that a set of $R$
   lines intersect at $x = x_{0}$, we set $a = 1$ and $b = x_{0}$.
\end{proof}

\subsection{Test about $R$ multivariate simple linear model under matrix multivariate elliptical model}\label{ssec:3}

In order to consider phenomena and experiments  under more flexible and robust conditions than the usual
normality,  various  works have appeared in the statistical literature since the 80's. Those efforts  has
been collected in various books and papers which are consolidated in the so termed generalised
multivariate analysis or multivariate statistics analysis under elliptically contoured distributions, see
\citet{gv:93} y \citet{fz:90}, among other authors. These new techniques generalize the classical matrix
multivariate normal distribution by a robust family of matrix multivariate distributions with elliptical
contours.

Recall that $\mathbf{Y} \in \Re^{n\times m}$ has a matrix multivariate elliptically contoured
distribution if its density with respect to the Lebesgue measure is given by:
$$
  dF_{\mathbf{Y}}(\mathbf{Y})=\frac{1}{|\mathbf{\Sigma}|^{n/2}|\mathbf{\Theta}|^{m/2}}
  h\left\{\tr\left[(\mathbf{Y}-\boldsymbol{\mu})'\mathbf{\Theta}^{-1}(\mathbf{Y}-
  \boldsymbol{\mu})\mathbf{\Sigma}^{-1}\right]\right\} (d\mathbf{Y}),
$$
where  $\boldsymbol{\mu} \in \Re^{n\times m}$, $ \mathbf{\Sigma} \in \Re^{m\times m}$, $ \mathbf{\Theta}
\in \Re^{n\times n}$, $\mathbf{\Sigma}>\mathbf{0}$ and $ \mathbf{\Theta}> \mathbf{0}$ and $(d\mathbf{Y})$
is the Lebesgue measure. The function $h: \Re \rightarrow [0,\infty)$ is termed the generator function
and satisfies $\int_{0}^\infty u^{mn-1}h(u^2)du < \infty$. Such a distribution is denoted by
$\mathbf{Y}\sim \mathcal{E}_{n\times m}(\boldsymbol{\mu},\mathbf{\Theta} \otimes \mathbf{\Sigma}, h)$,
\citet{gv:93}. Observe that this class of matrix multivariate distributions includes normal, contaminated
normal, Pearson type II and VI, Kotz, logistic, power exponential, and so on; these distributions have
tails that are weighted more or less, and/or they have greater or smaller degree of kurtosis than the
normal distribution.

Among other properties of this family of distributions, the invariance of some test statistics under this
family of distributions stands out, that is, some test statistics have the same distribution under
normality as under the whole family of elliptically contoured distributions, see theorems 5.3.3 and 5.3.4
of \citet[pp. 185-186]{gv:93}. Therefore, the distributions of Wilks, Roy, Lawley-Hotelling and Pillai
test statistics are invariant under the whole family of elliptically contoured distributions, see
\citet[pp. 297-299]{gv:93}.

Finally, note that, in multivariate linear model, it was assumed that the $x$'s were fixed. However, in
many applications, the $x$'s are random variables. Then, as in the normal case, see \citet[Section 10.8,
p. 358]{re:95}, if we assume that $\left (y_{1}, y_{2}, \dots, y_{q}, x \right )$ has a multivariate
elliptically contoured distribution, then all estimations and tests have the same formulation as in the
fixed-$x$, case. Thus there is no essential difference in our procedures between the fixed-$x$ case and
the random-$x$ case.

\section{Application}\label{sec:4}

The rosebush (\textit{Rosa sp. L.}) is the ornamental species of major importance in the \textit{State of
Mexico, Mexico}, being the red spider (\textit{Tetranychus urticae Koch}) (\textit{Acari: Tetranychidae})
its main entomological problem, the control has been based almost exclusively using acaricide, which has
caused this plague to acquire resistance in a short time. In order to counteract this problem in part, an
experiment was carried out using the variety of \textit{red petals Vega} in two greenhouses located in
the \textit{Ejido\footnote{A piece of land farmed communally, pasture land, other uncultivated lands, and
the fundo legal, or town site, under a system supported by the state.} "Los Morales", in Tenancingo,
State of Mexico, Mexico}, from October 2008 to August 2009. In a greenhouse, chemical control was applied
exclusively, while in the other, combined control (chemical and biological) was used, where applications
of acaricide were reduced and releases of two predatory mites were made: \textit{Phytoseiulus persimilis
Athias-Henriot} and \textit{Neoseiulus californicus McGregor} (\textit{Acari: Phytoseiidae}). The red
spider infestations decrease the length of the stem ($Y_{1}$) and the size of the floral button
($Y_{2}$), preponderant characteristics so that the final product reaches the best commercial value, so
that a total of 15 stems were measured randomly and weekly from each greenhouse, their respective floral
button, to quantify their length and diameter in centimeters, respectively, for a total of 15 weeks
($X$), see \cite{pr:14}. The measurements of the variables were carried out from January to April 2009
and the application of the treatments was initiated in week 44 of 2008.

The investigator considers\footnote{In the original work, the analysis was made based on univariate
statistical techniques only.} that a multivariate simple linear model for the results of each greenhouse
is the appropriate model to relate the two dependent variables $Y_{1}$ and $Y_{2}$ in terms of the
independent variable $X$. The corresponding multivariate simple linear models are
$$
  \mathbf{Y}_{r} = \mathbf{X}_{r} \bg{\beta}_{r} + \bg{\epsilon}_{r}, \quad r = 1, 2
$$
$\bg{\epsilon}_{r} \sim \mathcal{E}_{n_{r}\times 2}(\mathbf{0}, \mathbf{I}_{n_{r}} \otimes \bg{\Sigma})$,
$\bg{\Sigma}\in \Re^{2 \times 2} $, $\bg{\Sigma} > \mathbf{0}$, with $n_{1} = 15$, and $n_{2} = 15$ and
$$
  \bg{\beta}_{r} = \left (
        \begin{array}{cc}
            \alpha_{r1} & \alpha_{r2} \\
            \beta_{r1} & \beta_{r2} \\
        \end{array}
        \right) =
        \left (
        \begin{array}{c}
            \bg{\alpha}'_{r}\\
            \bg{\beta}'_{r} \\
        \end{array}
        \right)
$$
The researcher ask for the following hypotheses testing.
\begin{description}
  \item[i)] $H_{01}: \bg{\beta}_{1} = \bg{\beta}_{2}$, that is, the set of lines
  are parallel (if the average stem length and the average floral button diameter of each sample of roses per
  week are the same under the two methods of pest control);
  \item[ii)] $H_{02}: \bg{\alpha}_{1} = \bg{\alpha}_{2}$, that is, the set of lines
  have a common vector intercept (if the average stem length and the average floral button diameter
  of each sample roses in week zero are the same under the two methods of pest control).
\end{description}
The results of the experiment are presented in the next Table \ref{tab1}.
\begin{table}[!h]
  \caption{Experimental results of length of the stem (cms) and the diameter of the floral button (cms)
  of Vega rose variety.}\label{tab1}
\begin{center}
  \begin{tabular}{||r r r|r r r||}
    \hline\hline
    \multicolumn{3}{||c|}{Biological control} & \multicolumn{3}{c|}{Chemical control}\\
    $X$ & $Y_{1}$ & $Y_{2}$ & $X$ & $Y_{1}$ & $Y_{2}$ \\
    \hline\hline
    1  & 67.32 & 4.87 &  1 & 55.74 & 4.82 \\
    2  & 68.92 & 4.89 &  2 & 58.63 & 4.97 \\
    3  & 69.33 & 5.07 &  3 & 61.14 & 5.01 \\
    4  & 71.66 & 5.19 &  4 & 62.46 & 5.06 \\
    5  & 72.26 & 5.26 &  5 & 62.96 & 5.13 \\
    6  & 76.55 & 5.73 &  6 & 64.55 & 5.22 \\
    7  & 81.41 & 5.82 &  7 & 66.87 & 5.28 \\
    8  & 82.71 & 6.09 &  8 & 67.93 & 5.34 \\
    9  & 83.09 & 6.15 &  9 & 68.38 & 5.37 \\
    10 & 83.59 & 6.17 & 10 & 68.88 & 5.39 \\
    11 & 83.91 & 6.24 & 11 & 69.76 & 5.40 \\
    12 & 84.67 & 6.30 & 12 & 71.31 & 5.42 \\
    13 & 85.34 & 6.33 & 13 & 72.98 & 5.54 \\
    14 & 87.41 & 6.61 & 14 & 74.33 & 5.65 \\
    15 & 88.21 & 6.62 & 15 & 76.44 & 5.74 \\
    \hline
  \end{tabular}
\end{center}
\end{table}

Thus the matrices $\bg{\beta}_{1}$, $\bg{\beta}_{2}$ and $\mathbf{S}_{E}$ are given by
$$
  \bg{\beta}_{1} = \left (
                        \begin{array}{cc}
                          66.521429 & 4.75238095\\
                           1.571321 & 0.13378571\\
                        \end{array}
                   \right ), \quad
  \bg{\beta}_{2} = \left (
                        \begin{array}{cc}
                          56.416286 & 4.83647619\\
                           1.300964 & 0.05660714\\
                        \end{array}
                   \right )
$$
and
$$
  \mathbf{S}_{E} = \left (
                        \begin{array}{rr}
                          65.625451 & 3.9069754\\
                           3.906975 & 0.3025506\\
                        \end{array}
                   \right ).
$$
Moreover,
\begin{description}
  \item[i)] from Theorem \ref{teo3} \textbf{ii)} we have
  $$
    \mathbf{S}_{H} = \left (
                        \begin{array}{rr}
                          10.233018 & 2.9212089\\
                           2.921209 & 0.8339145\\
                        \end{array}
                   \right ), \quad\mbox{ and}
  $$
  \begin{table}[!h]
  \caption{Four criteria to proof $H_{01}: \bg{\beta}_{1} = \bg{\beta}_{2}$}\label{tab2}
  \medskip
  \begin{minipage}[t]{6in}
  \hspace{.5cm}
  \begin{tabular}{lll}
    \hline
      Criteria & Statistic & $\alpha$ Critical value \\
      \hline
      Wilks\footnote{\scriptsize Remember that for this tests, the decision rule is:
      statistics $\leq$ critical value} \phantom{jjjjjjjjjjjjjjjjjjjj}& 0.1159631
      \phantom{jjjjjjjjjjjjjjjjjj}& 0.860199 \\
      Roy & 0.8840369 & 0.775 \\
      Pillai & 0.8840369 & 0.775\\
      Lawley-Hotelling & 7.62343 & 4.225201\footnote{\scriptsize Using an F
      approximation, see equation (6.26) in \citet[p.166, 1995]{re:95}.
    }\\
   \hline
   \end{tabular}
   \end{minipage}
   \end{table}

   Thus, from Table \ref{tab2}, there is no doubt that the four criterions reject the null hypothesis
   $H_{01}: \bg{\beta}_{1} = \bg{\beta}_{2}$ for $\alpha = 0.05$.
  \item[ii)] Similarly, from Theorem \ref{teo3}\textbf{i)}, the matrix $\mathbf{S}_{\mathbf{H}}$ is given by
  $$
    \mathbf{S}_{H} = \left (
                        \begin{array}{rr}
                          172.934851 & -1.43916793\\
                           -1.439168 &  0.01197679\\
                        \end{array}
                   \right ).
  $$
 and
  \begin{table}[!h]
  \caption{Four criteria to proof $H_{02}: \bg{\alpha}_{1} = \bg{\alpha}_{2}$}\label{tab3}
  \medskip
  \begin{minipage}[t]{6in}
  \hspace{.5cm}
  \begin{tabular}{lll}
    \hline
      Criteria & Statistic & $\alpha$ Critical value \\
      \hline
      Wilks\footnote{\scriptsize Remember that for this tests, the decision rule is:
      statistics $\leq$ critical value} \phantom{jjjjjjjjjjjjjjjjjjjj}& 0.06658425
      \phantom{jjjjjjjjjjjjjjjjjj}& 0.860199 \\
      Roy & 0.9334158 & 0.808619 \\
      Pillai & 0.9334158 & 0.808619\\
      Lawley-Hotelling & 14.01857 & 4.225201\footnote{\scriptsize Using an F
      approximation, see equation (6.26) in \citet[p.166, 1995]{re:95}.
    }\\
   \hline
   \end{tabular}
   \end{minipage}
   \end{table}

   From Table \ref{tab3} we can conclude that under the four criterions of test the hypothesis
   $H_{02}: \bg{\alpha}_{1} = \bg{\alpha}_{2}$ is rejected for a level of significance of $\alpha = 0.05$.
\end{description}

\begin{figure}[ht!]
\begin{minipage}[b]{0.5\textwidth}
  \begin{center}
  \subfigure[Length of the stem]{\includegraphics[scale=0.4]{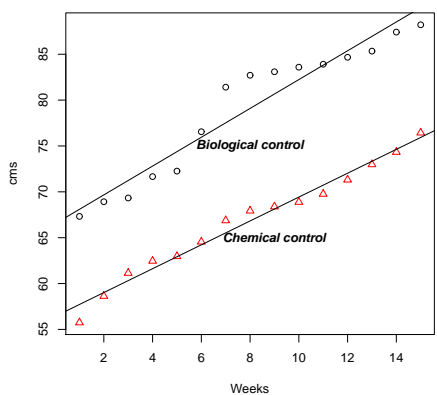}}
  \end{center}
\end{minipage}
\hfill
\begin{minipage}[b]{0.5\textwidth}
\centering \subfigure[Diameter of the floral button]{\includegraphics[scale=0.4]{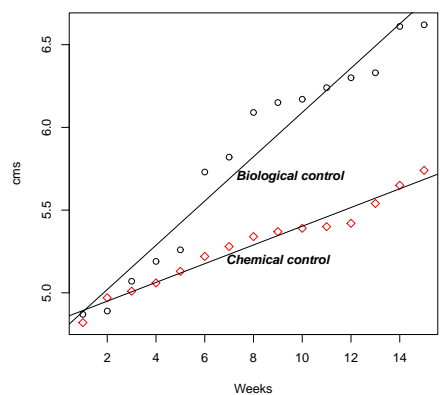}}
\end{minipage}
\caption{Observations and adjusted values}
\end{figure}

Given that $R = 2$, we can easily check graphically  the conclusions reached in the hypothesis testing.
Figure 1 (b) shows the intersection of lines for the floral button diameters, which explains the
rejection of parallelism hypothesis. However,  Figure 1(a) shows parallel  lines, which certainly implies
that the average length of stem for each sample per week is the same for both pest control. Similarly,
Figure 1(a) depicts very different  intercepts associated to the length of the stem, explaining the
rejecting of the hypothesis for equal intercepts. Also, Figure 1(b) shows equal intercepts, which implies
that the average floral button diameter for each sample in week zero is the same for both pest control.

%Now, from a practical point of view, we can establish the following conclusions:
%\begin{itemize}
% \item Difference of the average stem length at the week zero (66.52 cm) under the biological control with
%     the corresponding chemical control (56.41 cm) is due to the effect of the control method used.
%  \item Difference of the average floral button diameter at week zero (4.75 cm) under the biological control
%    with the corresponding  chemical control (4.83 cm), although minimal is due to the effect of the control
%    method used.
% Las siguientes hay que corregirlas
%  \item The increasing per week of the average stem length is of (1.57 cm/week) under the biological control differs
%    statistically of the corresponding issue under the chemical control (1.30 cm/week).
%\item The average increasing per week of the floral button (0.13 cm/week) under the biological control differs
%    statistically of the corresponding issue under the chemical control (0.05 cm/week).
%\end{itemize}

The thesis \citet{pr:14} concludes that the biological control method reduces infestation of the pest and
as a consequence both the stem length and the button size are increased. This aspect promotes a higher
sale price, but this result was not incorporated in the addressed work. Our analysis confirms these
conclusions, but in a robust way that include all the variables simultaneously.

\section{Conclusions}\label{sec:5}

As a consequence of Subsection \ref{ssec:3} the three hypotheses testing of this paper are valid under
the complete  family of elliptically contoured distribution, i.e. in any practical circumstance we can
assume that our information have a matrix multivariate elliptically  contoured distribution instead of
considering the usual non realistic normality.

%\section*{Acknowledgements}

%The authors wish to thank the Editor and the anonymous reviewers for their constructive comments
%on the preliminary version of this paper.

\end{document}